\documentclass[a4paper, 11pt
,reqno
]{amsart}

\usepackage{amsmath, amsthm, amssymb, mathrsfs, longtable, multirow}

\numberwithin{equation}{section}
\theoremstyle{plain}
 \newtheorem{thm}{Theorem}[section]
 
 \newtheorem{lem}[thm]{Lemma}
 \newtheorem{prop}[thm]{Proposition}
\theoremstyle{definition}
 
 \newtheorem{exmp}[thm]{Example}
 
\theoremstyle{remark}
 \newtheorem{rem}[thm]{Remark}

\calclayout

\allowdisplaybreaks[3]

\title[Hypergeometric function of type $A$]%
{The hypergeometric function 
for 
the root system of type $A$ with a certain degenerate parameter}
\author[N.~Shimeno]{Nobukazu Shimeno}
\address{School of Science and Technology, Kwansei Gakuin University, 
2-1 Gakuen, Sanda, Hyogo 669-1337, Japan}
\email{shimeno@kwansei.ac.jp}
\author[Y.~Tamaoka]{Yuichi Tamaoka}
\address{Graduate School of Science and Technology, Kwansei Gakuin University, 
2-1 Gakuen, Sanda, Hyogo 669-1337, Japan}
\email{eeo61784@kwansei.ac.jp}

\date{}

\begin{document}
\maketitle
\thispagestyle{empty}

\begin{abstract}
We express explicitly the Heckman-Opdam hypergeometric function for the root system 
of type $A$ with a certain degenerate parameter in terms of the 
 Lauricella hypergeometric function. 
\end{abstract}

\section*{Introduction}

Radial parts of zonal spherical functions on real semisimple Lie groups give a class of multivariable 
hypergeometric functions (\cite{HC1}, \cite{Hel2}). In rank one cases they are expressed by the 
Gauss hypergeometric function (\cite{Hel2}, \cite{Koornwinder}). 
Heckman and Opdam develop the theory of hypergeometric functions associated with root systems 
by generalizing zonal spherical functions (\cite{HO}, \cite{Hec}, \cite{Op:book}). 

On the other hand, generalizations of the classical hypergeometric functions of one-variable 
include hypergeometric series given by Appell, Lauricella, and Kamp\' e de F\'eriet, and 
the hypergeometric function of matrix argument (\cite{AK,IKSY, L}, \cite{BO}). 

It is of interest to identify different approaches to hypergeometric functions in 
several variables. Sekiguchi \cite{Se0, Se} shows that the zonal spherical function on 
$SL(n,\mathbb{R})$ for a certain degenerate parameter can be written by the 
Lauricella's hypergeometric function $F_D$. Tamaoka~\cite{T} shows that the Jack polynomial 
with a certain degenerate parameter can be written by Lauricella's $F_D$ (see Theorem~\ref{thm:main0} 
of this paper). 

Beerends \cite[Theorem~5.4]{B} shows that 
 the hypergeometric function associated 
with the root system of type $BC$ with a certain degenerate parameter can be written by 
the generalized  Kamp\'e de F\'eriet function. 
Beerends and Opdam \cite[Theorem~4.2]{BO} give a precise relation between  
the hypergeometric function of matrix argument and the hypergeometric function associated 
with the root system of type $BC$ with  a certain degenerate parameter. 

In the present paper we express explicitly the hypergeometric function associated with the root system 
of type $A$ with a certain degenerate parameter in terms of Lauricella's $F_D$ (Theorem~\ref{thm:main}). 
This result is indicated in \cite{S} without proof. It can be regarded as a 
generalization of a result of Sekiguchi \cite{Se0, Se} for 
a zonal spherical function on $SL(n,\mathbb{R})$ 
and that of Tamaoka \cite{T} for the Jack polynomial. Though our main result 
(Theorem~\ref{thm:main}) might be known 
for specialists, the result has never been stated in this precise form in the literature as far as we know. 

This paper is organized as follows. In \S1 we review on the hypergeometric function associated with the 
root system of type $A$  and define a certain degenerate parameter. In \S2 we prove our main theorem 
by using a system of differential equations of second order. 
We remark on the case of the Jack polynomial in \S3 and 
 relate our second order differential operators with trigonometric Dunkl operators in \S4.

\section{Hypergeometric function for the root system of type $A_{n-1}$}\label{sect:hg}

In a series of papers starting from \cite{HO}, Heckman and Opdam develop the theory of  the hypergeometric 
function associated with a root system. 
In this section we review on hypergeometric function for the root system of type $A_{n-1}$. 
We refer \cite{Hec}, \cite{Op:book} for details. 

Let $n$ be a positive integer greater than $1$ and equip $\mathbb{R}^n$ with the standard inner product $(\,\,,\,\,)$. 
Consider the subspace $\mathfrak{a}$ of $\mathbb{R}^n$ defined by 
\[
\mathfrak{a}=\{t\in\mathbb{R}^n\,:\,t_1+\cdots+t_n=0\}.
\]
We identify $\mathfrak{a}^*$ with $\mathfrak{a}$ by the inner product  $(\,\,,\,\,)$. 
Let $e_i$ denote the element of $\mathbb{R}^n$ with $i$-th entry $1$ and all the other entries $0$. 
Define
\[
R=\{e_j-e_i\,:\,1\leq i\not=j\leq n\},
\]
which a root system of type $A_{n-1}$. 
The Weyl group $W$ for $R$ is isomorphic to $S_n$. 

Let 
$\mathfrak{a}_\mathbb{C}^*$ denote the space of complex-valued linear functions on $\mathfrak{a}$. 
By the correspondence $\mathfrak{a}_\mathbb{C}^*\ni \lambda=\lambda_1e_1+\cdots+\lambda_n e_n \mapsto 
(\lambda_1,\dots,\lambda_n)\in \mathbb{C}^n$, we have the following identification:
\[
\mathfrak{a}_\mathbb{C}^*\simeq \{(\lambda_1,\dots,\lambda_n)\in\mathbb{C}^n:\lambda_1+\cdots+\lambda_n=0\}.
\]
The Weyl group $W\simeq S_n$ acts on $\mathfrak{a}_\mathbb{C}^*$ by 
$w\lambda=(\lambda_{w^{-1}(1)},\dots,\lambda_{w^{-1}(n)})$ $(w\in W,\,\lambda\in\mathfrak{a}_\mathbb{C}^*)$. 

Let $R_+$ denote the set of positive roots defined by
\[
R_+=\{e_j-e_i\,:\,1\leq i<j\leq n\}.
\]
Let $P_+$ denote the set of dominant integral weights
\begin{align*}
P_+&=\{\lambda\in\mathfrak{a}_\mathbb{C}^*: (\lambda,\alpha^\vee)\in \mathbb{Z}_+\,\,(\alpha\in R_+)\} \\
& \simeq \{(\lambda_1,\dots,\lambda_n): \lambda_1+\cdots+\lambda_n=0,\,\,\lambda_j-\lambda_i\in\mathbb{Z}_+\,\,
(1\leq i<j\leq n)\}.
\end{align*}
Here $\mathbb{Z}_+$ is the set of non-negative integers and 
$\alpha^\vee$ is the coroot of $\alpha$ defined by  $\alpha^\vee=2\alpha/(\alpha,\alpha)$. 
For $k\in\mathbb{C}$, define 
\begin{equation}\label{eqn:rho}
\rho(k)=\frac{k}{2}\sum_{\alpha\in R_+}\alpha
=\left(-\frac{n-1}{2}k,-\frac{n-3}{2}k,\dots,\frac{n-1}{2}k\right).
\end{equation}

Put
\[
 A=\exp\mathfrak{a}=\{z\in \mathbb{R}_{>0}^n\,:\,z_1\cdots z_n=1\}\subset \mathbb{R}_{>0}^n
\]
and $\vartheta_i=z_i\displaystyle\frac{\partial}{\partial z_i}$ $(1\leq i\leq n)$. 
The Weyl group $W\simeq S_n$ acts on $A$ by
\[
w(z_1,\dots,z_n)=(z_{w^{-1}(1)},\dots,z_{w^{-1}(n)}).
\]
We employ ``GL''-picture for convenience. We consider a function $\phi$ on $\mathbb{R}_{>0}^n$ 
 and impose the differential equation
\[
(\vartheta_1+\cdots+\vartheta_n)\phi=0
\]
to give a function on $A$. 
For $\lambda\in\mathfrak{a}_\mathbb{C}^*$ write 
$z^\lambda=z_1^{\lambda_1}\cdots z_n^{\lambda_n}$. 

Define the differential operator $L(k)$ by 
\begin{equation}\label{eqn:lk}
L(k)=\sum_{i=1}^n \vartheta_i^2+k\sum_{1\leq i<j\leq n}\frac{z_i+z_j}{z_i-z_j}(\vartheta_i-\vartheta_j).
\end{equation}
There exist a commutative algebra $\mathbb{D}(k)$ 
of  $W$-invariant differential operators containing $L(k)$ and an algebra 
isomorphism $\gamma:\mathbb{D}(k)\rightarrow S(\mathfrak{a}_\mathbb{C})^W$, 
the set of $W$-invariant elements of the symmetric algebra of $\mathfrak{a}_\mathbb{C}
$. 
In particular, we have $\gamma(L(k))(\lambda)=(\lambda,\lambda)-(\rho(k),\rho(k))$. 

\begin{rem}
If $k=1/2$, then $\mathbb{D}(k)$ is the set of the radial parts of invariant differential operators 
on $SL(n,\mathbb{R})/SO(n)$. For general $k$, the commutative algebra $\mathbb{D}(k)$ was 
first constructed by Sekiguchi \cite{Se0, Se2} giving a set of generators explicitly. 
In \S\ref{sect:dunkl}, we review Cherednik's construction of $\mathbb{D}(k)$ by 
the trigonometric Dunkl operators. 
\end{rem}

Let $Q$ be the $\mathbb{Z}$-span of $R$ and let 
$Q_+$ be the $\mathbb{Z}_+$-span of $R_+$. 
There exists a unique solution $\varphi(z)=\Phi(\lambda,k;z)$ on $A_+
:=\{z\in A\,:\,z_1<z_2<\cdots<z_n\}$ for
\begin{equation}\label{eqn:hoseco}
L(k)\varphi=
((\lambda,\lambda)-(\rho(k),\rho(k)))\varphi
\end{equation}
of the form
\begin{equation}
\label{eqn:hcser}
\Phi(\lambda,k;z)=\sum_{\mu\in Q_+}
\Gamma_\mu(\lambda,k)z^{\lambda-\rho(k)-\mu},
\quad \Gamma_0(\lambda,k)=0, 
\end{equation}
where the coefficients $\Gamma_\mu(\lambda,k)$ are rational function in 
$\lambda$ with possible poles 
at the hyperplane $H_{\mu}$ for some $\mu<0$, with
\[
H_\mu=\{\lambda\in\mathfrak{a}_\mathbb{C}^*\,:\,(2\lambda+\mu,\mu)=0\}.
\]
Moreover, $\varphi=\Phi(\lambda,k)$ also satisfies
\begin{equation}
\label{eqn:hgs}
D\varphi=\gamma(D)(\lambda)\varphi\quad (D\in \mathbb{D}(k)).
\end{equation}
From \cite[Proposition 4.2.5]{Hec} the apparent simple pole of $\Phi(\lambda,k)$ along $H_\mu$ is removable 
unless $\mu=n\alpha$ for some $n\in -\mathbb{Z}_+$ and $\alpha\in R_+$. 
We call $\lambda\in\mathfrak{a}_\mathbb{C}^*$ generic if $(\lambda,\alpha^\vee)\not\in\mathbb{Z}$ for all 
$\alpha\in R$. 
$\lambda=(\lambda_1,\dots,\lambda_n)\in \mathfrak{a}_\mathbb{C}^*$ is generic if and only if 
$\lambda_i-\lambda_j\not\in\mathbb{Z}\,\,(1\leq i<j\leq n)$. 
If $\lambda\in\mathfrak{a}_\mathbb{C}^*$ is generic, then 
\[
\{\Phi(w\lambda,k)\,:\,w\in W\}
\]
forms a basis of the solution space of (\ref{eqn:hgs}) on $A_+$ (\cite[Corollary 4.2.6]{Hec}). 

Let $\Gamma(\,\cdot\,)$ denote the Gamma function. 
Let $\tilde{c}(\lambda,k)$ and $c(\lambda,k)$ denote the meromorphic functions defined by 
\begin{align}
& \tilde{c}(\lambda,k)=\prod_{1\leq i<j\leq n}\frac{\Gamma(\lambda_j-\lambda_i)}{\Gamma({\lambda_j-\lambda_i}+k)}, \label{eqn:tcf}\\
& c(\lambda,k)=\frac{\tilde{c}(\lambda,k)}{\tilde{c}(\rho(k),k)}. \label{eqn:cf}
\end{align}
For $k=1/2,\,1,\,2$, $c(\lambda,k)$ agrees with Gindikin-Karpelevich's product formula for
 Harish-Chandra's $c$-function on $G/K=SL(n,\mathbb{K})/SO(n,\mathbb{K})$, 
where $\mathbb{K}=\mathbb{R},\,\mathbb{C},\,\mathbb{H}$, respectively. 

The denominator in (\ref{eqn:cf}) is given explicitly as follows. 
\[
\tilde{c}(\rho(k),k)=\prod_{1\leq i<j\leq n}\frac{\Gamma((j-i)k)}{\Gamma((j-i+1)k)}=\prod_{j=2}^n\left(\frac{\Gamma(k)}{\Gamma(jk)}\right).
\]
Let $S$ denote the set of poles of of $1/\tilde{c}(\rho(k),k)$, that is
\begin{equation}\label{eqn:singk}
S=\{k\in\mathbb{C}\setminus\mathbb{Z}_{<0}\,:\,jk\in \mathbb{Z}_{<0}\text{ for some }j=2,3,\dots, n\}.
\end{equation}
For $k\in \mathbb{C}\setminus S$ and generic $\lambda$ define
\begin{equation}\label{eq:connection}
F(\lambda,k)=\sum_{w\in W}c(w\lambda,k)\Phi(w\lambda,k).
\end{equation}
Heckman and Opdam proved that $\varphi(z)=F(\lambda,k;z)$ extends to an entire 
function of $\lambda\in\mathfrak{a}_\mathbb{C}^*,\,k\in \mathbb{C}\setminus S$ and 
$z$ in a tubular neighborhood of $A$ in $A_\mathbb{C}$ and 
is a unique $W$-invariant 
real analytic solution of 
(\ref{eqn:hgs})  on $A$ such that $\phi(\mathbf{1})=1$ (\cite[Part I, Chapter 4]{Hec}, \cite[\S 6.3]{Op:book}, 
\cite[Corollary 4.8]{OS}). 

If $k=1/2,\,1$, or $2$, then $F(\lambda,k)$ is the restriction to $A$ of the zonal spherical function 
on $G/K=SL(n,\mathbb{K})/SO(n,\mathbb{K})$ with $\mathbb{K}=\mathbb{R},\,\mathbb{C}$, or $
\mathbb{H}$, respectively. Here $A$ is the maximally split abelian subgroup of $G$ with the 
Cartan decomposition $G=KAK$. 

If $k\geq 0$ and $\mu\in P_+$, then 
from \cite[(4.4.10)]{Hec} we have 
\begin{equation}\label{eqn:hojacobi}
F(\mu+\rho(k),k)=c(\mu+\rho(k),k)P(\mu,k), 
\end{equation}
where $P(\mu,k)$ is the Jacobi polynomial of Heckman and Opdam.

Let $n$ be an integer greater than $1$. 
For $\nu\in\mathbb{C}$ define $\lambda(\nu,k)\in\mathfrak{a}_\mathbb{C}^*$ by
\begin{equation}\label{eqn:degpar}
 \lambda(\nu,k)=\left(-\frac{\nu}{n},\cdots,-\frac{\nu}{n},\frac{(n-1)\nu}{n}\right)+\rho(k).
\end{equation}
It follows from the definition that 
\begin{equation}\label{eqn:cfdeg}
c(\lambda(\nu,k),k)=\frac{\Gamma(nk)\Gamma(\nu+k)}{\Gamma(k)\Gamma(\nu+nk)}.
\end{equation}
For $\nu\in\mathbb{Z}_+$ it can be written by the shifted factorial
\begin{equation}\label{eqn:cfdeg2}
c(\lambda(\nu,k),k)=\frac{(k)_\nu}{(nk)_\nu}.
\end{equation}
Here the shifted factorial is defined by $(a)_0=1$ and $(a)_n=a(a+1)\cdots (a+n-1)$ for $n\in\mathbb{Z}_{>0}$. 

Notice that $\lambda(\nu,k)$ is generic if and only if 
\begin{equation}\label{eqn:generic}
pk\not\in \mathbb{Z}\,\,(1\leq p\leq n-2)\,\,\text{ and }\,\,\nu+qk\not\in\mathbb{Z}\,\,(1\leq q\leq n-1).
\end{equation}

Let $W_\Theta$ denote the permutation group of $\{1,\dots,n-1\}$ and $W^\Theta$ denote the 
set of representatives of minimal length for the coset $W_\Theta\backslash W$. That is, $W^\Theta$ consists of the elements
\[
w_1=
e
,\,\,
w_2=\begin{pmatrix}1 & 2 & \cdots & n-2& n-1 & n \\ 1 & 2 & \cdots & n-2 & n & n-1\end{pmatrix},\dots,\,\,
w_n=\begin{pmatrix}1 & 2 & 3 & \cdots & n \\ n & 1 & 2 & \cdots & n-1\end{pmatrix}.
\]
Here $w_i\,\,(1\leq i\leq n)$ is the element of $S_n$ of minimal length such that $w_i(n+1-i)=n$. 

We have the following proposition for the hypergeometric function with the degenerate parameter. 

\begin{prop}\label{prop:deghs}
\emph{(i)}  
Assume $k\in \mathbb{C}\setminus 
S$ and $\lambda(\nu,k)$ is generic \emph{(\eqref{eqn:singk}, \eqref{eqn:generic})}. 
Then 
\[
F(\lambda(\nu,k),k;z)=\sum_{i=1}^n c(w_i\lambda(\nu,k),k)\Phi(w_i\lambda(\nu,k),k;z) \quad
 (z\in A_+).
\]
\emph{(ii)} Assume $\nu\in \mathbb{Z}_+$ and $qk\not\in \mathbb{Z}_{<0}$ for any $1\leq q\leq n$. Then 
\[
F(\lambda(\nu,k),k;z)=
\frac{(k)_\nu}{(nk)_\nu}
\Phi(\lambda(\nu,k),k;z)\quad (z\in A). 
\]
\end{prop}
\begin{proof}
First we prove (1). The case of $k=0$ is trivial. In addition to the assumption of the proposition, 
we assume $k\not=0$. 
If $w\in W\,\setminus \,W^\Theta$, then there exists $l\,\,(1\leq l\leq n-2)$ such that $w(l)>w(l+1)$. 
Then $w\lambda(\nu,k)_j-w\lambda(\nu,k)_i=\lambda(\nu,k)_l-\lambda(\nu,k)_{l+1}=-k$ for $j=w(l)$ and 
$i=w(l+1)$. By the definition \eqref{eqn:tcf} and \eqref{eqn:cf} of the $c$-function, $c(w\lambda(\nu,k),k)=0$ unless 
$w\in W^\Theta$. 

It follows from the definition that
\[
c(w_i\lambda(\nu,k),k)=\frac{\Gamma(nk)\Gamma(\nu+ik)\Gamma(-\nu-(i-1)k)}{\Gamma(k)\Gamma(\nu+nk)\Gamma(-\nu)}
\quad (2\leq i\leq n). 
\]
Thus $c(w_i\lambda(\nu,k),k)=0\,\,(2\leq i\leq n)$ for $\nu\in\mathbb{Z}_+$ and 
the equation of (2) holds on $A_+$ under the assumption of (1). 
Notice that $\lambda(\nu,k)-\rho(k)\in P_+$ if and only if $\nu\in \mathbb{Z}_+$. 
Thus we have $\Phi(\lambda(\nu,k),k)=P(\lambda(\nu,k)-\rho(k),k)$ and (2) follows from 
\eqref{eqn:hojacobi} and \eqref{eqn:cfdeg} by analytic 
continuation.  
\end{proof}

In \S\ref{sect:second} we will show that $F(\lambda(\nu,k),k)$ can be written by the 
Lauricella hypergeometric function $F_D$.

\section{A system of hypergeometric differential equations of second order}\label{sect:second}
Let $n$ be an integer greater than $1$ and $k$ be a complex number. Let
$z=(z_1,z_2,\dots,z_n)$ denote a variable in $\mathbb{R}^n$ and put 
$\displaystyle\vartheta_i=z_i\frac{\partial}{\partial z_i}\,\,(1\leq i\leq n)$. 
For $1\leq i<j\leq n$ define differential operator $\Delta_{ij}$ by
\begin{equation}\label{eqn:dij}
\Delta_{ij}=\vartheta_i\vartheta_j-\frac{k}{2}\frac{z_i+z_j}{z_i-z_j}(\vartheta_i-\vartheta_j)+\left(\frac{\nu}{n}+\frac{k}{2}\right)(\vartheta_i+\vartheta_j).
\end{equation}
We consider the system of differential equations
\begin{align}
& (\vartheta_1+\cdots+\vartheta_n)\varphi=0, \label{eqn:euler} \\
&  \Delta_{ij}\phi=-\frac{\nu(\nu+nk)}{n^2}\varphi \quad  (1\leq i<j\leq n). \label{eqn:deq1}
\end{align}

By summing up (\ref{eqn:deq1}) for $1\leq i<j\leq n$ and using (\ref{eqn:euler}) we have
\begin{equation}\label{eqn:casimir}
\left(\sum_{1\leq i<j\leq n}\vartheta_i\vartheta_j-\frac{k}{2}\frac{z_i+z_j}{z_i-z_j}(\vartheta_i-\vartheta_j)\right)\varphi
=-\frac{(n-1)\nu(\nu+nk)}{2n}\varphi
\end{equation}
The differential operator in the left hand side of (\ref{eqn:casimir}) is 
equivalent to $-\frac12$ times the second order hypergeometric differential operator 
$L(k)$ given in (\ref{eqn:lk})  
and the coefficients of $\varphi$ in the right hand side of (\ref{eqn:casimir}) is 
$-\frac12\{(\lambda(\nu,k),\lambda((\nu,k))-(\rho(k),\rho(k))\}$. 
Hence (\ref{eqn:casimir}) is equivalent to (\ref{eqn:hoseco}) with 
$\lambda=\lambda(\nu,k)$. 

If $k=1/2$, then (\ref{eqn:deq1}) are radial parts of the differential equations satisfied by the zonal spherical function on $G=SL(n,\mathbb{R})$ 
that is the Poisson integral of a $SO(n)$-invariant section of a degenerate principal 
series representation on $G/P_\Theta$, where $P_\Theta$ is a maximal parabolic subgroup of $G$ whose Levi part is 
isomorphic to $GL(n-1,\mathbb{R})$. These differential equations on $G$ are 
given by Oshima \cite{O} 
using generalized Capelli operators in $U(\mathfrak{gl}(n,\mathbb{C}))$. Moreover, (\ref{eqn:casimir}) is the radial part of the differential equation 
corresponding to the Casimir operator. 

%

By the change of variables 
\begin{equation}\label{eqn:change}
y_i=\frac{z_i}{z_n}\quad (1\leq i\leq n-1),\quad y_n=z_n, 
\end{equation}
we have
\[
\vartheta_i=y_i\frac{\partial}{\partial y_i}\quad (1\leq i\leq n-1),\quad \vartheta_1+\cdots+\vartheta_n=y_n\frac{\partial}{y_n}.
\]
Hence, (\ref{eqn:euler}) means that $\varphi$ does not depend on $y_n$. Operators (\ref{eqn:dij}) become
\begin{align}
& \Delta_{ij}=\vartheta_i\vartheta_j-\frac{k}{2}\frac{y_i+y_j}{y_i-y_j}(\vartheta_i-\vartheta_j)+\left(\frac{\nu}{n}+\frac{k}{2}\right)(\vartheta_i+\vartheta_j)
\quad (1\leq i<j\leq n-1), \label{eqn:dij2}\\
& \Delta_{in}=\vartheta_i(\vartheta_1+\cdots+\vartheta_{n-1})-\frac{k}{2}\frac{y_i+1}{y_i-1}(\vartheta_i+\vartheta_1+\cdots+\vartheta_{n-1}) \notag \\
& \phantom{aaaaaaaaaaaaaaaaaaa}
+\left(\frac{\nu}{n}+\frac{k}{2}\right)(\vartheta_i-(\vartheta_1+\cdots+\vartheta_{n-1}))\quad (1\leq i\leq n-1). \label{eqn:dij3}
\end{align}
Here we write $\vartheta_i=y_i\frac{\partial}{\partial y_i}\,\,(1\leq i\leq n-1)$ in (\ref{eqn:dij2}) and (\ref{eqn:dij3}). 
Putting 
\begin{equation}\label{eqn:gauge}
\varphi(y_1,\dots,y_{n-1})=(y_1\cdots y_{n-1})^{-\frac{\nu}{n}}u(y_1,\dots,y_{n-1}),
\end{equation}
the differential equations (\ref{eqn:deq1}) become the following differential 
equations for $u$:
\begin{align}
& \left(\vartheta_i\vartheta_j-k\frac{y_j\vartheta_i-y_i\vartheta_j}{y_i-y_j}\right)u=0\quad (1\leq i<j\leq n-1), \label{eqn:deq2}\\
& \left(-\left(\vartheta_1+\cdots+\vartheta_{n-1}-\nu+\frac{k}{y_i-1}\right)\vartheta_i \right. \notag \\
& \phantom{aaaaaaaaaaaaaaa}\left.
-\frac{ky_i}{y_i-1}(\vartheta_1+\cdots+\vartheta_{n-1}-\nu)\right)u=0\quad (1\leq i\leq n-1), \label{eqn:deq3} 
\end{align}
which can be written in the following form:
\begin{align}
& y_i(\vartheta_i+k)\vartheta_ju=y_j(\vartheta_j+k)\vartheta_iu\quad (1\leq i<j\leq n) \label{eqn:lauricella1}, \\
 & \vartheta_i(\vartheta_1+\cdots+\vartheta_{n-1}-\nu-k)u \notag \\ 
& \phantom{aaaaaa}=y_i(\vartheta_i+k)(\vartheta_1+\cdots+\vartheta_{n-1}-\nu)u 
\quad (1\leq i\leq n-1). \label{eqn:lauricella2}
\end{align}

We recall Lauricella's $F_D$ of $n-1$ variables and the corresponding system $E_D$ of differential equations of rank $n$. 
Lauricella's hypergeometric function $F_D$ is the analytic continuation of the series
\begin{align}
F_D & (\alpha,\beta_1,\dots,\beta_{n-1},\gamma;\,y_1,\dots,y_{n-1}) \notag \\
&\phantom{aaaa}=\sum_{m_1,\dots,m_{n-1}\geq 0}
\frac{(\alpha)_{m_1+\cdots+m_{n-1}}(\beta_1)_{m_1}\cdots (\beta_{n-1})_{m_{n-1}}}{(\gamma)_{m_1+\cdots+m_{n-1}}m_1!\cdots m_{n-1}!}
y_1^{m_1}\cdots y_{n-1}^{m_{n-1}}, \label{eqn:lauricellas}
\end{align}
where $\alpha,\,\beta_1,\dots,\beta_{n-1},\,\gamma$ are complex constants 
with $\gamma\not=-1,-2,\dots$. 
It satisfies the following system of differential equations:
\begin{equation}\tag{$E_D$}
\left\{\begin{alignedat}{1}
& y_i(\vartheta_i+\beta_i)\vartheta_j F=y_j(\vartheta_j+\beta_j)\vartheta_i F\,\, (1\leq i<j\leq n), \\
&  \vartheta_i(\vartheta_1+\cdots +\vartheta_{n-1}+\gamma-1)F \\ 
& \phantom{aaaaaaaaaaaa}=y_i(\vartheta_i+\beta_i)(\vartheta_1+\cdots+\vartheta_{n-1}+\alpha)F
\,\, (1\leq i<j\leq n).
\end{alignedat}\right.
\end{equation}
The system $(E_D)$ is holonomic of rank $n$. If $\gamma\not=-1,-2,\dots$, then 
Lauricella's hypergeometric function $F_D(\alpha,\beta_1,\dots,
\beta_{n-1},\gamma;y_1,\dots,y_{n-1})$ is the unique analytic solution 
of $(E_D)$ such that $F(0)=1$. 
We refer \cite{L}, \cite[\S 9.1]{IKSY}, and \cite{MS} on  
Lauricella's $F_D$. 

Equations (\ref{eqn:lauricella1}) and (\ref{eqn:lauricella2}) constitute $(E_D)$ 
with 
\begin{equation}\label{eqn:abc}
\alpha=-\nu,\quad\beta_1=\cdots \beta_{n-1}=k,\quad \gamma =-\nu-k+1.
\end{equation}

By the change of variables
\begin{equation}\label{eqn:cv1}
x_i=1-y_i\quad (1\leq i\leq n), 
\end{equation}
(\ref{eqn:lauricella1}) give equations of the same form
\begin{equation}
x_i(\vartheta_i+k)\vartheta_j u=x_j(\vartheta_j+k)\vartheta_i u\quad (1\leq i\leq n). \label{eqn:lauricella3}
\end{equation}
Here we write $\vartheta_i=x_i\frac{\partial}{\partial x_i}\,\,(1\leq i\leq n)$. 
By (\ref{eqn:lauricella2}) and (\ref{eqn:lauricella3}), we have the following equations.
\begin{equation}
\vartheta_i(\vartheta_1+\cdots+\vartheta_{n-1}+nk-1)u=x_i(\vartheta_i+k)(\vartheta_1+\cdots+\vartheta_{n-1}-\nu)u
\quad (1\leq i\leq n). \label{eqn:lauricella4}
\end{equation}
Equations (\ref{eqn:lauricella3}) and (\ref{eqn:lauricella4}) constitute $(E_D)$ with 
\begin{equation}
\alpha=-\nu,\quad \beta_1=\cdots=\beta_{n-1}=k,\quad \gamma=nk.
\end{equation}
Consequently, if $nk\not=-1,-2,\dots$, then 
\begin{equation}\label{eqn:symsol1}
\varphi(y)=(y_1\cdots y_{n-1})^{-\frac{\nu}{n}}F_D(
-\nu,k,\dots,k,nk;\,1-y_1,\dots,1-y_{n-1})
\end{equation}
is the unique analytic solution for (\ref{eqn:euler}) and (\ref{eqn:deq1}) 
satisfying $\varphi(1,\dots,1)=1$. 

The symmetric group $S_n$ acts on the variable $z=(z_1,\dots,z_n)$ as permutations, hence it acts on the variable $y=(y_1,\dots,y_{n-1})$. 

\begin{lem}
The function $\varphi$ in \eqref{eqn:symsol1} is $S_n$-invariant. 
\end{lem}
\begin{proof}
Recall that $y_i=z_i/z_n$ for $1\leq i\leq n-1$. 
The transposition $(i,i+1)\,\,(1\leq i\leq n-2)$ interchanges $y_i$ and $y_{i+1}$. Since $\beta_1=\cdots=\beta_{n-1}=k$, 
$\varphi(y)$ is invariant under the transposition  $(i,i+1)\,\,(1\leq i\leq n-2)$.

By the transposition $(n-1,n)$, $y_1,\dots,y_{n-2},y_{n-1}$ change to $y_1/y_{n-1},\dots,y_{n-2}/y_{n-1}$, $1/y_{n-1}$, respectively. 
It follows from the transformation formula (\cite[p 149]{L})
\begin{align*}
F_D & (\alpha,  \beta_1,  \dots,\beta_{n-1},\gamma;\,x_1, \dots,x_{n-1})
 =(1-x_{n-1})^{-\alpha} \\
& 
\times F_D  \left(\alpha,\beta_1,\dots,\beta_{n-2},\gamma-\beta_1-\cdots-\beta_{n-1},\gamma; 
 \frac{x_{n-1}-x_1}{x_{n-1}-1},\dots,\frac{x_{n-1}-x_{n-2}}{x_{n-1}-1},\frac{x_{n-1}}{x_{n-1}-1}\right)
\end{align*}
that $\varphi(y)$ is invariant under the transposition $(n-1,n)$. Since $S_n$ is generated by
 the transpositions $(1,2),\,(2,3),\dots,(n-1,n)$, the lemma is proved. 
\end{proof}

Now we state the main result of this paper, which asserts that the hypergeometric function 
of type $A_{n-1}$ with parameter $\lambda(\nu,k)$ defined by  (\ref{eqn:degpar}) is 
written by Lauricella's $F_D$ as in  (\ref{eqn:symsol1}). 
In the case of $k=1/2$, that is the case of 
the zonal spherical function on $SL(n,\mathbb{R})/SO(n)$, this theorem is given 
by Sekiguchi \cite{Se0, Se}. 

\begin{thm}\label{thm:main}
Assume $k\in\mathbb{C}\setminus S$ and $\nu\in\mathbb{C}$. Then we have
\[
F(\lambda(\nu,k),k;z)=
(y_1\cdots y_{n-1})^{-\frac{\nu}{n}}F_D(
-\nu,k,\dots,k,nk;\,1-y_1,\dots,1-y_{n-1}), 
\]
where $\lambda(\nu,k)$ and $y_i$ are given by \eqref{eqn:degpar} and \eqref{eqn:change}. 
\end{thm}
\begin{proof}
For $2\leq j\leq n$, 
there exists a unique series solution $u_j(y)$ with the leading term 
$\prod_{i=n-j+2}^{n-1}y_i^{-k}y_{n-j+1}^{(j-1)k+\nu}$ for the system $(E_D)$ with \eqref{eqn:abc} 
that converges on a neighbourhood of $y=0$ in 
$\{y\in \mathbb{R}^{n-1}\,:\,0<y_1<y_2<\cdots <y_{n-1}\}$. Moreover, the set of 
$u_1(y):=F_D(-\nu,k,\dots,k,-\nu-k+1;y)$ and $u_j(y)\,\,(2\leq j\leq n)$ forms a basis of 
local solutions of $E_D$ for generic $k$ and $\nu$ (\cite[Section~3.3.1 (f)]{GKZ}, 
\cite[Section 1.5]{SST}, \cite[Section 5]{goto}). 

For $1\leq i\leq n$, let $\varphi_i$ denote the solution of the system of equations \eqref{eqn:euler} and \eqref{eqn:deq1} corresponding to $u_i$ by \eqref{eqn:change} and \eqref{eqn:gauge}. 
Then $\varphi_i$ is a solution of \eqref{eqn:casimir} with the characteristic exponent 
$w_i\lambda(\nu,k)-\rho(k)$ and the leading coefficient $1$. Thus $\varphi_i=\Phi(w_i\lambda(\nu,k),k)$ and 
it is a solution of the hypergeometric system  \eqref{eqn:hgs} with $\lambda=\lambda(\nu,k)$. 
Since $\{\varphi_i\,:\,1\leq i\leq n\}$ forms a basis of the solution space of  the system of equations \eqref{eqn:euler} and \eqref{eqn:deq1} for generic $k$ and $\nu$, \eqref{eqn:symsol1} 
is a solution of  \eqref{eqn:hgs} with $\lambda=\lambda(\nu,k)$ for any $k\in\mathbb{C}\setminus S$ and 
$\nu\in\mathbb{C}$ by analytic continuation.

Thus \eqref{eqn:symsol1} is a $S_n$ invariant solution of \eqref{eqn:hgs} with $\lambda=\lambda(\nu,k)$ 
that is real analytic and $\varphi(\mathbf{1})=1$. 
Hence $\varphi(z)=F(\lambda(\nu,k),k,z)$ by the uniqueness of the 
hypergeometric function.  
\end{proof}

\begin{exmp}
We give examples of $A_1$ and $A_2$. Assume $k\in\mathbb{C}\setminus S$. 

First we consider the case of $A_1$. Lauricella's $F_D$ of one variable is the Gauss 
hypergeometric function ${}_2 F_1$. From Proposition~\ref{prop:deghs} and Theorem~\ref{thm:main}, it holds on $A_+$ that 
\begin{align*}
F(\lambda & (\nu,k),k;z)=
y_1^{-\frac{\nu}{2}}{}_2 F_1(-\nu,k,2k;1-y_1) \\
 = & \frac{\Gamma(2k)\Gamma(\nu+k)}{\Gamma(k)\Gamma(\nu+2k)}y_1^{-\frac{\nu}{2}}{}_2F_1(-\nu,k,-\nu-k+1;y_1)  \\
& +\frac{\Gamma(2k)\Gamma(-\nu+k)}{\Gamma(k)\Gamma(-\nu)}y_1^{\frac{\nu}{2}+k}{}_2 F_1(\nu+2k,k,\nu+k+1;y_1).
\end{align*}
The above equalities are well-known formulae for the hypergeometric function of type $A_1$ and the 
Gauss hypergeometric function (\cite[Example 6.3]{Op:book}, \cite[proof of Theorem~4.3.6]{Hec}, \cite{Bateman}). 
Note that $F(\lambda(\nu,k),k)$ can be written by the Jacobi function (\cite{Koornwinder})
\begin{equation}\label{eqn:jacobi}
F(\lambda(\nu,k),k;z)=\phi_{2\sqrt{-1}(\nu+k)}^{(k-1/2,k-1/2)}\left(\frac{t}{2}\right):=
{}_2 F_1\left(-\nu,\nu+2k,k+\frac12;-\sinh^2\frac{t}{2}\right), 
\end{equation}
where $z=(e^t,e^{-t})$ and $y_1=e^{2t}$. If $\nu\in\mathbb{Z}_+$, then
\[
F(\lambda(\nu,k),k;z)=\frac{(k)_\nu}{(2k)_\nu}y_1^{-\frac{\nu}{2}}{}_2F_1(-\nu,k,-\nu-k+1;y_1)
\]
and from \eqref{eqn:jacobi}
\[
F(\lambda(\nu,k),k;z)=\frac{\nu !}{(2k)_\nu}C_\nu^{(k)}(\cosh t), 
\]
where $C_\nu^{(k)}$ denote the Gegenbauer polynomial (\cite[\S 3.15.1]{Bateman}). 

Next we consider the case of $A_2$. 
Lauricella's $F_D$ of two variables is  Appell's $F_1$ function. 
From Proposition~\ref{prop:deghs} and Theorem~\ref{thm:main}, it holds on $A_+$ that 
\begin{align*}
F(\lambda & (\nu,k),k;z)=
(y_1y_2)^{-\frac{\nu}{3}}  F_1(-\nu,k,k,3k;1-y_1,1-y_2) \\
 = &\frac{\Gamma(3k)\Gamma(\nu+k)}{\Gamma(k)\Gamma(\nu+3k)}(y_1y_2)^{-\frac{\nu}{3}}
F_1(-\nu,k,k,-\nu-k+1;y_1,y_2) \\
&+\frac{\Gamma(3k)\Gamma(\nu+2k)\Gamma(-\nu-k)}{\Gamma(k)\Gamma(\nu+3k)\Gamma(-\nu)}y_1^{\frac{2\nu}{3}+k}y_2^{-\frac{\nu}{3}}
G_2(k,k,\nu+2k,-\nu-k;-y_1/y_2,-y_2) \\
& + \frac{\Gamma(3k)\Gamma(-\nu-2k)}{\Gamma(k)\Gamma(-\nu)}(y_1^{-2}y_2)^{-\frac{\nu}{3}-k}
F_1(\nu+3k,k,k,\nu+2k+1;y_1/y_2,y_1). 
\end{align*}
Here  
\[
G_2(\alpha,\alpha',\beta,\beta',x,y)=
\sum_{m,n=0}^\infty (\alpha)_m(\alpha')_n(\beta)_{n-m} (\beta')_{m-n}\frac{x^my^n}{m!n!}, 
\]
where $(\alpha)_n=\Gamma(\alpha+n)/\Gamma(\alpha)$. 
The above equality among Appell's $F_1$ and $G_2$ functions is a special case of 
\cite[(19)]{Olsson}. 
If $\nu\in\mathbb{Z}_+$, then
\[
F(\lambda(\nu,k),k;z)=\frac{(k)_\nu}{(3k)_\nu}(y_1y_2)^{-\frac{\nu}{3}}
F_1(-\nu,k,k,-\nu-k+1;y_1,y_2)
\]
or
\[
P(\lambda(\nu,k)-\rho(k),k;z)=(y_1y_2)^{-\frac{\nu}{3}}
F_1(-\nu,k,k,-\nu-k+1;y_1,y_2), 
\]
where $P(\lambda(\nu,k)-\rho(k),k;z)$ is the Jacobi polynomial of Heckman and Opdam. 

\end{exmp}

\section{
The case of $\nu\in\mathbb{Z}_+$
}

In this section assume that 
$k>0$ and $\nu\in\mathbb{Z}_+$. 
Put
\begin{equation}
\mu(\nu)=\left(-\frac{\nu}{n},\cdots,-\frac{\nu}{n},\frac{(n-1)\nu}{n}\right).
\end{equation}
Then 
it follows from Proposition~\ref{prop:deghs} and Theorem~\ref{thm:main} that 
\begin{equation*}
P(\mu(\nu),k;z)=\frac{(nk)_\nu}{(k)_\nu}(y_1\cdots y_{n-1})^{-\frac{\nu}{n}}F_D(
-\nu,k,\dots,k,nk;\,1-y_1,\dots,1-y_{n-1})
\end{equation*}
and
\begin{equation*}
P(\mu(\nu),k;z)=(y_1\cdots y_{n-1})^{-\frac{\nu}{n}}F_D(
-\nu,k,\dots,k,nk;y_1,\dots,y_{n-1}),
\end{equation*}
where $P(\mu(\nu),k)$ is the Jacobi polynomial of Heckman and Opdam. 

The Jacobi polynomial for the root system of type $A_{n-1}$ is essentially 
the Jack polynomial. 
A partition $\lambda$ of length equal or less than $n$ is a sequence 
$\lambda=(\lambda_1,\dots,\lambda_n)$ of nonnegative integers such that 
$\lambda_1\geq \lambda_2\geq \cdots \lambda_n\geq 0$. Define $|\lambda|=\sum_{i=1}^n \lambda_i$. 
For two partitions $\lambda$ and $\mu$ we write $\mu\leq \lambda$ if 
$|\mu|=|\lambda|$ and $\sum_{i=1}^j \mu_j\leq \sum_{i=1}^j \lambda_i$ for all $j\geq 1$. 

For a partition $\lambda$ of length equal or less than $n$ define the monomial symmetric function 
$m_\lambda$ by
\[
m_\lambda=\sum_{\alpha\in S_n\lambda}x^\alpha.
\]
There exists a unique $P_\lambda^{(1/k)}$ that satisfies the following conditions:
\begin{align}
& P_\lambda^{(1/k)}=\sum_{\mu\leq \lambda}v_{\lambda\mu}m_\mu\qquad v_{\lambda\mu}\in\mathbb{C}(k),\qquad 
v_{\lambda\lambda}=1, \label{eqn:jack1}\\
& L(k)P_\lambda^{(1/k)}=h(\lambda)P_\lambda^{(1/k)},\qquad
h(\lambda)=\sum_{i=1}^n \lambda_i(\lambda_i+k(n+1-2i)).\label{eqn:jack2}
\end{align}
We call $J^{(1/k)}_\lambda(z)$ the Jack polynomial  (\cite{Mac0, Mac}).
From \cite[Proposition 3.3]{BO} we have
\begin{equation}\label{eqn:hojack}
P_\lambda^{(1/k)}(z)=P(\pi(\lambda),k;z)\quad (z\in A)
\end{equation}
for a partition $\lambda=(\lambda_1,\dots,\lambda_n)$ of length equal or less than $n$. 
Here $\pi(\lambda)\in P_+
$ is given by
\[
\pi(\lambda)=\sum_{i=1}^n\lambda_i e_i-\frac1n\left(\sum_{i=1}^n\lambda_i \right)
\left(\sum_{i=1}^n e_i\right).
\]


Thus we have the following result as a corollary of Theorem~\ref{thm:main}. 


\begin{thm}[Tamaoka~\cite{T}]\label{thm:main0}
Assume $k>0$, $p,\,q\in \mathbb{Z}_+$ and $p\geq q$. Then for $z\in\mathbb{C}^n$
\begin{align*}
P_{(p,q,\cdots,q)}^{(1/k)}(z_{1},\cdots ,z_{n})
& =\frac{(nk)_{p-q}}{(k)_{p-q}}
\prod_{i=1}^{n-1}z_{i}^{q}z_{n}^{p}
F_{D}\left(q-p,k,\cdots ,k,nk;
1-\frac{z_{1}}{z_{n}},\cdots,1-\frac{z_{n-1}}{z_{n}}\right)\\
& =\prod_{i=1}^{n-1}z_{i}^{q}z_{n}^{p}
F_{D}\left(q-p,k,\cdots ,k,q-p-k+1;
\frac{z_{1}}{z_{n}},\cdots,\frac{z_{n-1}}{z_{n}}\right).
\end{align*}
\end{thm}

\begin{rem}
One of the authors~\cite{T} 
proves Theorem~\ref{thm:main0} 
without using the Heckman-Opdam theory. He just used  the characterization 
of the Jack polynomial by the conditions \eqref{eqn:jack1}, \eqref{eqn:jack2} and 
properties of the Lauricella hypergeometric function. 

In view of \eqref{eqn:hojacobi} and \eqref{eqn:hojack}, Theorem~\ref{thm:main0} asserts that 
Theorem~\ref{thm:main} holds for $k>0$ and $\nu\in\mathbb{Z}_+$. We can deduce 
Theorem~\ref{thm:main} from 
Theorem~\ref{thm:main0} in the same manner as the proof of \cite[Theorem~4.2]{BO}. 
\end{rem}

\section{
Trigonometric Dunkl operators
}
\label{sect:dunkl}

First we review the trigonometric Dunkl operator or the Cherednik operator in the $GL_n$ case (\cite{C1, Op:book}). 
For $1\leq i\leq n$, define the trigonometric Dunkl operator ${T}_i$ (\cite[\S 3.5]{C1}) by
\[
{T}_i=\vartheta_i+k\sum_{i<j}\frac{z_i}{z_i-z_j}(1-\sigma_{ij})+k\sum_{i>j}\frac{z_j}{z_i-z_j}(1-\sigma_{ij})
+\rho(k)_i.
\]
Here $\sigma_{ij}$ is the permutation $(i \,j)$ that acts as the transposition of the  coordinates $z_i$ and $z_j$. 
Notice that the choice of the positive system of $R$ to define the trigonometric Dunkl operators is opposite to 
that of \cite[\S 3.5]{C1}. 
We take $-R_+$ as the positive system of $R$ to define $T_i$. 

The Cherednik operators satisfy the following relations:
\begin{align}
& [T_i,T_j]=0\quad (1\leq i,\,j\leq n), \quad \sigma_{i}T_j=T_j \sigma_{i}\quad (j\not=i,\,i+1),\label{eqn:r1}\\
& \sigma_{i}T_i-T_{i+1}\sigma_{i}=-k, \label{eqn:r2}
\end{align}
where $\sigma_i=\sigma_{i\,i+1}$ $(1\leq i\leq n-1)$. 
(\ref{eqn:r1}) and (\ref{eqn:r2}) are the defining relations of the degenerate affine Hecke algebra 
$
\mathbf{H}
=\langle \mathbb{C}S_n,x_1,\dots,x_n\rangle$ of 
type $GL_n$, if we replace $T_i$ by $x_i$ in the above relations. 

For $p\in S(\mathfrak{a}_\mathbb{C})$, put
\[
T_p=p(T_1,\dots,T_n).
\]
Write 
\[
T_p=\sum_{w\in W}D_w^{(p)}w, 
\]
where $D_w^{(p)}\,\,(w\in W)$ are differential operator on $A$ and 
define the differential operator $D_p$ on $A$ by
\[
D_p=\sum_{w\in W}D_w^{(p)}.
\]
$D_p$ is the differential operator that has the same restriction to symmetric functions as $T_p$. 

For $1\leq m\leq n$, let $e_m(x)$ denote the $m$-th elementary symmetric polynomial:
\[
e_m(x)=\sum_{1\leq i_1<\cdots<i_m\leq n}x_{i_1}\cdots x_{i_m}.
\]
Then we have
\[
D_{e_1}=\vartheta_1+\cdots+\vartheta_n
\]
and
\[
D_{e_2}=\sum_{1\leq i<j\leq n}\vartheta_i\vartheta_j-\frac{k}{2}\frac{z_i+z_j}{z_i-z_j}(\vartheta_i-\vartheta_j)
-\frac{k^2}{4}\binom{n+1}{3}, 
\]
which are differential operators in (\ref{eqn:euler}) and (\ref{eqn:casimir}), respectively. 
For $p(x)=x_1^2+\cdots +x_n^2$, $D_p=L(k)+(\rho(k),\rho(k))$, where $L(k)$ is defined in (\ref{eqn:lk}). 
The commutative algebra $\mathbb{D}(k)$ mentioned in \S1 is 
\[
\mathbb{D}(k)=\{D_p\,:\,p\in S(\mathfrak{a}_\mathbb{C})^W\}
\]
and is generated by $L_{e_1},\dots,L_{e_n}$. Moreover, the algebra isomorphism 
$\gamma:\mathbb{D}(k)\rightarrow S(\mathfrak{a}_\mathbb{C})^W$ mentioned in \S1 is 
defined by $D_p\mapsto p$. 

The following proposition asserts that the 
differential operator $\Delta_{ij}$ given in \S\ref{sect:second} (\ref{eqn:deq1}) 
is also related with the trigonometric Dunkl operators. 

\begin{prop}[\cite{S}]
For $1\leq i<j\leq n$, define $p_{ij}\in S(\mathfrak{a}_\mathbb{C})$ by 
\[
p_{ij}(x)=\left(x_i-\rho(k)_i+\frac{\nu}{n}\right)\left(x_j-\rho(k)_j+k+\frac{\nu}{n}\right).
\]
Then we have
\[
D_{p_{ij}}=\Delta_{ij}+\frac{\nu(\nu+nk)}{n^2}.
\]
\end{prop}

\begin{rem}
Though the choice of the positive system to define the trigonometric Dunkl operators is not essential, 
we choose as above because it matches better with the characteristic exponents $w_i\lambda(\nu,k)-\rho(k)$ $(1\leq i\leq n)$ 
given in \S\ref{sect:hg} and the indicial equation 
\begin{equation}\label{eqn:indeq}
\displaystyle\left(\mu_i-\rho(k)_i+\frac{\nu}{n}\right)\left(\mu_j-\rho(k)_j+k+\frac{\nu}{n}\right)=0.
\end{equation}
of (\ref{eqn:deq1}) at infinity on $A_+$. 
If 
\[
\nu\not=-k,-2k,\dots,(-n+1)k,
\]
then the set of common solutions for (\ref{eqn:indeq}) for $1\leq i<j\leq n$ is 
$\{w_1\lambda(\nu,k),\dots,w_n\lambda(\nu,k)\}$, where $\lambda(\nu,k)$ and $w_j$ are given in 
\S\ref{sect:hg}.

This fact can be regarded as a special case of \cite[Theorem 9, Equation (27), Theorem 22]{O2}. 
Indeed, in \cite{O, O2}, Oshima constructed generators of annihilators of generalized Verma modules for $\mathfrak{gl}_n$ 
by using generalized Capelli operators.   
The deformation parameter $\varepsilon$ in \cite{O2} corresponds to $k$ in this paper. 
Using results in \cite{O2}, some part of results in this paper can be generalized to the case of 
arbitrary $\Theta\subset\{1,2,\dots,n\}$ as indicated in \cite{S}. We will discuss in detail elsewhere. 
\end{rem}

\begin{rem}
After we have finished our work, we noticed that the system of differential equations 
\eqref{eqn:euler}, \eqref{eqn:deq1} and its characteristic exponents are stated in 
\cite{chl}. 
\cite[Theorem 3.3]{chl} asserts that the system \eqref{eqn:euler}, \eqref{eqn:deq1}  is of rank $n$ 
and its solutions are also solutions of the hypergeometric system \eqref{eqn:hgs} 
with $\lambda=\lambda(\nu,k)$ without proof. 
\end{rem}


\begin{thebibliography}{aa}
\bibitem{AK}P. Appell and J. Kamp\'e de F\'eriet, \emph{
Fonctions hyperg\'eom\'etriques et hypersph\'eriques: polynomes d'Hermite}, Paris: Gauthier-Villars, 1926.

\bibitem{B}R.J. Beerends, \emph{Some special values for the $BC$ type hypergeometric function}, 
Contemp. Math. \textbf{138} (1992), 27--49. 

\bibitem{BO} R.J. Beerends and E.M. Opdam, \emph{Certain hypergeometric series related to the root system $BC$}, 
Trans. Amer. Math. Soc. \textbf{339} (1993), 581--609.


\bibitem{C1}I. Cherednik, \emph{Lectures on Knizhnik-Zamolodchikov equations and Hecke algebras}, MSJ Memoirs
Volume 1, 1998, 1--96.
\bibitem{chl}W. Couwenberg, G. Heckman, and E. Looijenga, 
\emph{On the geometry of the Calogero-Moser system}, Indag. Math. N.S, \textbf{16} (2005), 443--459. 
\bibitem{Bateman}A. Erd\'elyi, ed., \emph{Higher Transcendental Functions, Vol. 1},  McGraw Hill, New York, 1953. 
\bibitem{GKZ}I.M. Gel'fand, A.V. Zelevinsky, and M.M. Kapranov, \emph{Hypergeometric functions and toral
manifolds}, Funct. Anal. Appl. \textbf{23} (1989), 94--106.
\bibitem{goto}Y. Goto, \emph{Contiguity relations of Lauricella's $F_D$ revisited}, Tohoku Math. J. 
\textbf{69} (2017), 287--304. 

\bibitem{HC1} Harish-Chandra,
{\it Spherical functions on a semisimple Lie group I},
Amer. J. Math. {\bf 80} (1958), 241--310.

\bibitem{HO} G.J. Heckman and E.M. Opdam,
\emph{Root systems and hypergeometric functions I},
Comp. Math. {\bf 64} (1987), 329--352.

\bibitem{Hec}G.J. Heckman, \emph{Hypergeometric and Spherical Functions}, 
In:~\emph{Harmonic Analysis and Special Functions on Symmetric Spaces},
Perspect. Math., Academic Press, Boston, MA, 1994.

\bibitem{Hel2} S. Helgason, \emph{Groups and Geometric Analysis}, 
Amer. Math. Soc., 2000, c1984.


\bibitem{IKSY}K. Iwasaki, H. Kimura, S. Shimomura, and M. Yoshida, 
\emph{From Gauss to Painlev\'e: A Modern Theory of Special Functions}, Springer, 1991. 


\bibitem{Koornwinder}T. Koornwinder, \emph{Jacobi functions and analysis on noncompact semisimple Lie groups. } 
In: \emph{Special Functions: Group Theoretical Aspects and Applications}, 1--85, Math. Appl., 
Reidel, Dordrecht, 1984. 

\bibitem{L}G. Lauricella, \emph{Sulle Funzioni ipergeometrche a piu variabili}, Rend. Circ. Math. Palermo 
\textbf{7} (1893), 111-158. 


\bibitem{Mac0}I.G. Macdonald, \emph{Commuting differential operators and zonal spherical functions}, Algebraic groups, Utrecht 1986,  189--200, 
Lecture Notes in Math., \textbf{1271}, Springer, Berlin, 1987.

\bibitem{Mac}I.G. Macdonald, \emph{
Symmetric Functions and Hall Polynomials}, 2nd ed., Oxford University Press, 1995.


\bibitem{MS}K. Mimachi and T. Sasaki, \emph{Irreducibility and reducibility of Lauricella's system of 
differential equations $E_{D}$ and the Jordan-Pochhammer differential equation $E_{JP}$}, 
Kyushu J. Math. \textbf{66} (2012), 61--87. 

\bibitem{OS}
H. Oda and N. Shimeno, \emph{Spherical functions for small $K$-types}, arXiv:1710.02975.

\bibitem{Olsson}P.O.M. Olsson, \emph{Integration of the partial differential equations for the hypergeometric
functions $F_1$ and $F_D$ of two and more variables}, J. Math. Phys. \textbf{5} (1964), 420--430.


\bibitem{Op:book} E.M. Opdam,
\emph{Lecture notes on Dunkl operators for real and complex reflection groups},
MSJ Memoirs {\bf 8}. Mathematical Society of Japan, Tokyo, 2000.

\bibitem{O}T. Oshima, \emph{Generalized Capelli identities and boundary value problems for $GL(n)$}, Structure
of Solutions of Differential Equations, World Scientific, 1996, 307--335.

\bibitem{O2}T. Oshima, \emph{A quantization of conjugacy classes of matrices}, Adv. Math. \textbf{196} (2005), 124--146.

\bibitem{SST}M. Saito, B. Sturmfels, and N. Takayama, \emph{Gr\"obner Deformations 
of Hypergeometric Differential Equations}, Springer-Verlag, Berlin, 2000.

\bibitem{Se0}J. Sekiguchi, \emph{
Zonal spherical functions on the symmetric space $SL(3, \mathbb{R})/SO(3)$ and related topics} (in Japanese), Master Dissertation, Nagoya 
University, 1976.

\bibitem{Se}J. Sekiguchi, \emph{Zonal spherical functions on $SL(3,\mathbb{R})$} (in Japanese), RIMS Kokyuroku \textbf{266} (1976), 259--274. 

\bibitem{Se2}J. Sekiguchi, Zonal spherical functions on some symmetric spaces, Publ. RIMS. Kyoto Univ., \textbf{12} Suppl. (1977), 455--464.

\bibitem{S}N. Shimeno, \emph{Factorization of invariant differential equations}, Proceeding of the Symposium on Representation Theory 
held at Hadomisaki, Saga prefecture, Japan (1997), 54--58.


\bibitem{T}Y. Tamaoka, \emph{Jack polynomials and Lauricella's hypergeometric series} (in Japanese), Master Dissertation, 
Kwansei Gakuin University, 2018. 
\end{thebibliography}
\end{document}